\documentclass[11pt]{amsart}

\usepackage{amsmath,amsfonts,amscd,amssymb}

\theoremstyle{plain}
\def\theequation{\thesection.\arabic{equation}}
\newtheorem{theorem}[equation]{Theorem}
\newtheorem{conjecture}[equation]{Conjecture}
\newtheorem{proposition}[equation]{Proposition}
\newtheorem{lemma}[equation]{Lemma}
\newtheorem{corollary}[equation]{Corollary}
\theoremstyle{definition}
\newtheorem*{acknowledgements}{Acknowledgements}
\newtheorem{remark}[equation]{Remark}

\newtheorem{definition}[equation]{Definition}
\newtheorem{notation}[equation]{Notation}
\newtheorem{example}[equation]{Example}
\newtheorem{fact}[equation]{Fact}

\def\leftchoice#1#2#3#4{{\def\arraystretch{0.7}
\Bigl\{\!\!\begin{array}{ll}
   \scriptstyle #1,\!\!\!&\scriptstyle #2\cr
   \scriptstyle #3,\!\!\!&\scriptstyle #4\end{array}}}
\def\leftchoicestretch#1#2#3#4#5{{\def\arraystretch{#1}
\Bigl\{\!\!\begin{array}{ll}
   \scriptstyle #2\!\!\!&\scriptstyle #3\cr
   \scriptstyle #4\!\!\!&\scriptstyle #5\end{array}}}

\def\rksel#1#2#3{\rk_{#3}(#1/#2)}

\def\lara{\langle,\rangle}
\def\blangle{\boldsymbol{\langle}}
\def\brangle{\boldsymbol{\rangle}}

\let\iff\Leftrightarrow

\let\lar\longrightarrow
\let\iso\cong
\let\tensor\otimes

\def\beq{$$\begin{array}{llllllllllllllll}}
\def\eeq{\end{array}$$}
\def\beqn{\begin{equation}\begin{array}{llllllllllllllll}}
\def\eeqn{\end{array}\end{equation}}

\input cyracc.def       
\font\tencyr=wncyr10
\def\sha{\text{\tencyr Sh}}

\def\tbuildrel#1\over#2{\buildrel\text{\rm\normalsize\smaller[3]#1}\over{#2}}

\def\thincdots{\raise1.3pt\hbox{$\scriptscriptstyle
  \>\cdot\>\cdot\>\cdot\>\cdot\hskip0.3pt$}}

\def\IV{\text{\rm IV}}
\def\IVS{\text{\rm IV$^*$}}

\def\comment{}
\def\endcomment{}

\def\<{\raise0.5pt\hbox{$\,\scriptstyle<\,$}}

\catcode`\@=11

\def\bb@symb#1|#2{\expandafter\def\csname #2#1\endcsname{\mathbb{#1}}}
\def\bbsymbols#1#2{\@for\@tmpz:=#2\do{\expandafter\bb@symb\@tmpz|{#1}}}

\def\cal@symb#1|#2{\expandafter\def\csname #2#1\endcsname{\mathcal{#1}}}
\def\calsymbols#1#2{\@for\@tmpz:=#2\do{\expandafter\cal@symb\@tmpz|{#1}}}

\def\dmth@p#1|{\expandafter\let\csname#1\endcsname\relax
  \expandafter\DeclareMathOperator\csname#1\endcsname{#1}}
\def\operators#1{\@for\@tmpz:=#1\do{\expandafter\dmth@p\@tmpz|}}

\catcode`\@=12

\operators{Tr,rk,ord,Gal,BSD,Hom,tors,coker,Ind,Res,Frob,GL,PGL,SL,Aut}
\operators{Sel,End,Maps,Reg,Sy,Syl,sgn,id}
\DeclareMathOperator\vchar{char}

\def\Qp{\Q_p}
\def\Zp{\Z_p}

\bbsymbols{}{R,C,Z,N,Q,F}
\calsymbols{}{O}

\operators{rk,Orb,Stab,Rep}
\def\triv{{\mathbf 1}}
\def\m{{\mathfrak m}}
\csname@addtoreset\endcsname{equation}{section}
\def\OO#1#2#3{H_{#3}^{\scriptscriptstyle #1/#2}}
\def\neron#1#2{\omega_{v}^\circ}
\def\neronnov#1#2{\omega^\circ}
\def\E{{\mathcal E}}

\def\K{{\mathcal K}}
\def\KK{{\hbox{\tiny$\scriptscriptstyle\mathcal K$}}}
\def\PP{{\hbox{\tiny$\scriptscriptstyle P$}}}
\let\PPP\PP  
\def\KKP{{\hbox{\tiny$\scriptscriptstyle{\mathcal K}(P)$}}}
\def\KKR{{\hbox{\tiny$\scriptscriptstyle{\mathcal K}(\alpha)$}}}  
\def\FF{{\mathbb F}}
\def\F{{\mathcal F}}
\def\M{{\mathcal M}}
\def\L{{\mathcal L}}
\def\cH{{\mathcal H}}
\def\XX{{\mathcal X}}
\def\Xp#1#2{\XX_p(#1/#2)}
\def\X#1#2#3{\XX_{#1}(#2/#3)}
\def\IV{\text{\rm IV}}
\def\IVS{\text{\rm IV$^*$}}
\def\citeeq#1{\>\,{\buildrel{#1}\over=}\>\,}
\def\citeeqref#1{{\tbuildrel{{\ref{#1}}}\over=}}
\def\Tau{{\mathbf T}}
\def\Ttp{\Tau_{\Theta,p}}

\def\RC{{\mathcal C}}
\def\p{{\mathfrak p}}

\def\CCbig#1#2{C_{#1/#2}}
\def\CC#1#2{{C_{\hbox{\tiny$\scriptscriptstyle #1/#2$}}}}

\def\Cy{{\rm C}}
\def\Di{{\rm D}}

\def\Sym{{\rm S}}

\def\conjeq{\>\>{\buildrel?\over=}\>\>}
\let\qequal\conjeq

\def\daggerequation#1#2#3{\begingroup      
  \def\theequation{#2}
  \begin{equation}
  \label{#1}
    #3
  \end{equation}\endgroup
  \addtocounter{equation}{-1}}

\def\Esmslash{\hskip -0.15mmE\hskip -0.3mm/\hskip -0.3mm}
\def\smminusone{(\hskip -0.15mm\raise0.243em\hbox to 5pt{\hrulefill}\hskip 0.15mm1\hskip -0.15mm)}

\def\rksel#1#2#3{\rk_{#3}#1/#2}

\begin{document}

\comment

\title{Root numbers and parity of ranks of elliptic curves}
\author{Tim and Vladimir Dokchitser}
\dedicatory{to J.W.S. Cassels on the occasion of his 87th birthday}
\date{June 9, 2009}
\address{Robinson College, Cambridge CB3 9AN, United Kingdom}
\email{t.dokchitser@dpmms.cam.ac.uk}
\address{Gonville \& Caius College, Cambridge CB2 1TA, United Kingdom}
\email{v.dokchitser@dpmms.cam.ac.uk}

\begin{abstract}
The purpose of the paper is to complete several
global and local results concerning parity of ranks of elliptic curves.
Primarily, we show that the Shafarevich-Tate conjecture implies
the parity conjecture for all elliptic curves over number fields,
give a formula for local and global root numbers of elliptic curves
and complete the proof of a conjecture of Kramer and Tunnell
in characteristic 0.
The method is to settle the outstanding local formulae by deforming from
local fields to totally real number fields and then using global
parity results.
\end{abstract}

\maketitle

\setcounter{tocdepth}{1}
\tableofcontents

\section{Introduction}

The principal arithmetic invariant of an elliptic curve $E$
defined over a number field $K$ is its Mordell-Weil rank $\rk E/K$,
the rank of the group of rational points modulo torsion.
In view of the Birch--Swinnerton-Dyer conjecture,
its parity should be governed by another arithmetic invariant,
the global root number $w(E/K)=\pm 1$
(the expected sign in the conjectural functional equation for the
$L$-function of $E/K$):

\begin{conjecture}[Parity Conjecture]
$(-1)^{\rk E/K} = w(E/K).$
\end{conjecture}

The purpose of this paper is to clean up some known results
concerning this conjecture. Primarily, we show that it follows
from finiteness of the Tate-Shafarevich group,
and give an explicit formula for $w(E/K)$
and thus, conjecturally, for the parity of the rank:

\begin{theorem}[=Thm. \ref{main}]
\label{imain}
Let $E$ be an elliptic curve over a number field~$K$, and suppose
$\sha(E/K(E[2]))$ has finite 2- and 3-primary parts. Then
$$
  (-1)^{\rk E/K} = w(E/K).
$$
\end{theorem}

\begin{theorem}[=Thm. \ref{gloroot}]
\label{igloroot}
Let $E$ be an elliptic curve over a number field~$K$,
and set $F\!=\!K(E[2])$ and $d\!=\![F\!:\!K]$. Choose a non-trivial
\hbox{2-torsion} point $P$ of $E$, defined over $K$ if $d\le 2$,
and write $E'=E/\{O,P\}$ for the 2-isogenous curve
over $L=K(P)$. Then
\beq
  w(E/K) = \left\{
  \begin{array}{ll}
     (-1)^{\ord_2\frac{\CC EL\CC EF}{\CC{E'}L\CC{E'}F}
     +\ord_3\frac{\CC EF\CC EK^2}{\CC E{K(\sqrt{\Delta_E})}\CC EL^2}}&\>\>\>\text{if }d=6,\cr
     (-1)^{\ord_2\frac{\CC EL}{\CC{E'}L}}&\>\>\>\text{if }d<6.\cr
  \end{array}
  \right.
\eeq
\end{theorem}

\noindent
Here $\Delta_E$ is the discriminant of $E/K$,
and $\CCbig EK$ the product%
\footnote{
  Computational note:
  $\CCbig EK\!=\!\prod_{v} C(E/K_v,\omega)$ for any
  invariant differential $\omega$ for $E/K$. At finite places,
  $C(E/K_v,\omega)\!=\!c_v \cdot |\omega/\neron{E}{K_v}|^{}_v$,
  where $c_v$ is the local Tamagawa number at~$v$ and $\neron{E}{K_v}$ the
  minimal 
  differential at $v$, and hence can be computed with Tate's algorithm.
  The contributions from $v|\infty$ cancel in the
  $\ord_3$ term when using any $\omega$ for $E/K$.
  In the $\ord_2$ term,
  use an $L$-rational $\omega'$ for $E'$ and $\omega\!=\!\phi^*\omega'$ for $E$,
  where $\phi: E\!\to\! E'$ \hbox{is~the~2-isogeny;}
  then every $v|\infty$ where
  $\phi$ maps $E(L_v)$ onto $E'(L_v)$ contributes $-1$ to the $\ord_2$ term.
}
of the `local fudge factors' and periods in the conjectural Birch--Swinnerton-Dyer
formula for $L(E/K,s)$ at $s=1$, see \S\ref{ssnotation} for the precise
definition.

Theorem \ref{imain} is already known under mild technical restrictions
on $E$ (\cite{Squarity}~Thm. 1.3).
We remove the restrictions by a roundabout
approach that avoids unpleasant local computations at the troublesome
primes above~2 and~3.
This method of dealing with difficult reduction cases
applies to other parity-related local formulae.
In particular, we will use it to derive
a formula for local root numbers (analogous to Theorem \ref{igloroot})
and to prove the remaining case of the Kramer-Tunnell conjecture
for elliptic curves over $p$-adic fields.

Most of these local formulae have a global counterpart
related to Selmer groups.
Recall that
there is a 
version of the parity conjecture for the \hbox{$p^\infty$-Selmer}
rank $\rksel EKp$, the $\Zp$-corank of the $p^\infty$-Selmer group:

\begin{conjecture}[$p$-Parity Conjecture]
$(-1)^{\rksel EKp} = w(E/K).$
\end{conjecture}

\noindent
It is of course expected that $\rksel EKp=\rk E/K$ for each $p$,
as the difference is the number of copies of $\Q_p/\Z_p$ in the
(supposedly finite) group $\sha(E/K)$.
The cases that we address are
2-parity over quadratic extensions,
$p$-parity for curves with a $p$-isogeny when $p=2$ and 3, and
a slight improvement on~\cite{Tamroot} for $p$-parity results coming
from Brauer relations
and in $p$-adic towers.
See \S\ref{ssother} for precise statements.

\subsection{The method \& the Kramer--Tunnell conjecture}

There are several results that express the parity of a specific Selmer rank
in terms of purely local data of the elliptic curve
(e.g. \cite{CFKS,Squarity,Selfduality,FisA,Kra,MR}).
Now recall that the global root number is defined as the product
over all places of local root numbers,
$$
  w(E/K) = \prod_v w(E/K_v).
$$
Hence, if this `local data'
can be related to the corresponding local root number,
it can be used to prove the parity conjecture for this Selmer rank.

For example, Kramer's theorem \cite{Kra} Thm. 1
asserts that for a quadratic extension $F=K(\sqrt\alpha)$,
$$
  (-1)^{\rksel EF2} = \prod_v
  (-1)^{\dim_{\FF_2}\smash{\frac{E(K_v)}{N E(F_w)}}},
$$
the product taken over all places $v$ of $K$, choosing any $w|v$ of $F$.
Here \hbox{$N: E(F_w)\to E(K_v)$} is the norm map on points.
In \cite{KT} Kramer and Tunnell found the conjectural relation to the root number,
\daggerequation{iktloc}{$*$}
  {(-1)^{\dim_{\FF_2}\smash{\frac{E({K_v})}{N E({F_w})}}} =
  w(E/{K_v})w(E_\alpha/{K_v})(-\Delta_E,\alpha)_{K_v},}%
where
$E_\alpha$ is the quadratic twist of $E$ by $\alpha$, and
$(\cdot,\cdot)$ is the Hilbert symbol.
Granting this relation and taking the product over all $v$, we get
the 2-parity conjecture for $E/F$,
$$
  (-1)^{\rksel EF2} = w(E/F).
$$
(Observe that the Hilbert symbols disappear globally by the product formula.)
Kramer and Tunnell established \eqref{iktloc} in all cases but
the most horrible one, namely when $v|2$, $F_w/K_v$ is ramified
and $E/K_v$ has additive reduction.

We now explain how to prove \eqref{iktloc}
in the remaining cases (see \S\ref{skratun} for details).
Instead of confronting residue characteristic 2, we
use a `global-to-local' \linebreak approach.
The point is that the $p$-parity conjecture can be seen to hold
over totally real fields for all elliptic curves with non-integral
$j$-invariant, by considerations that exploit modularity
(Nekov\'a\v r \cite{NekIV} for odd $p$ and Theorem~\ref{2parthm} for $p=2$).
Therefore we can prove the local formula by reversing the argument
above:
if $K$ and $F$ are totally real and we know that \eqref{iktloc} holds
at all places $v$ but one, it must hold over the remaining place as well.
Because there is a sufficient supply of such fields and curves and
all the terms in \eqref{iktloc} are continuous (as functions of the
coefficients of the curve), it follows that the local formula must
always be true.

To summarise, we obtain:

\begin{theorem}[Kramer-Tunnell conjecture in characteristic 0]
Let $\K$ be a local field of characteristic zero,
$\F=\K(\sqrt\alpha)$ a quadratic extension,
$E/\K$ an elliptic curve and $E_\alpha/\K$ its quadratic twist by~$\alpha$.
Then
$$
  w(E/\K)w(E_\alpha/\K)(-\Delta_E,\alpha)_\K 
    = (-1)^{\dim_{\FF_2}\smash{\frac{E(\K)}{N E(\F)}}}.
$$
Here $\Delta_E$ is the discriminant of some model of $E/\K$,
$N: E(\F)\to E(\K)$ the norm map on points, and
$(\,\cdot\,,\,\cdot\,)_\K$ the Hilbert symbol in $\K$.
\end{theorem}

\begin{corollary}[2-parity in quadratic extensions]
Let $E$ be an elliptic curve over a number field $K$.
For every quadratic extension $F/K$,
$$
  (-1)^{\rksel EF2} = w(E/F).
$$
\end{corollary}

\subsection{Other parity results}
\label{ssother}

The following two theorems are slight extensions of known results
for general $p$ (see \S\ref{stotreal} and \S\ref{sisogroot}).

\begin{theorem}[see Thm. \ref{2parthm}]
Let $K$ be a totally real field, and $E/K$ an elliptic curve with
non-integral $j$-invariant. Then the $2$-parity conjecture holds for $E/K$.
\end{theorem}

\begin{theorem}[=Cor. \ref{isogroot23}]
Let $E$ be an elliptic curve over a number field~$K$ that
admits a $K$-rational isogeny of degree $p=2$ or 3. Then
$$
  (-1)^{\rksel EKp} = w(E/K).
$$
\end{theorem}

Next, we consider twists of elliptic curves by
Artin representations.
Recall that if $F/K$ is a Galois extension, the dual $p^\infty$-Selmer
$$
  \small
  \Xp EF=(\text{Pontryagin dual of the $p^\infty$-Selmer group of } E/F)\>\tensor\Q_p
$$
is a $\Qp$-rational representation of $G=\Gal(F/K)$.
One can compare multiplicities
$\blangle \tau, \Xp EF\brangle$ of $G$-representations in it
to the twisted root numbers $w(E/K,\tau)$; conjecturally,
$$
  \qquad\quad
  (-1)^{\blangle \tau, \Xp EF\brangle} = w(E/K,\tau)
     \qquad\qquad\!\! \text{($p$-parity conjecture for twists)},
$$
for every self-dual complex representation $\tau$ of $G$.
When $\tau=\C[G/H]$, this simply recovers the $p$-parity conjecture for
$E/F^H$.

Our main result concerning twists is somewhat technical
(Corollary \ref{gltamroot}),
but here are a few consequences (see \S\ref{sstamappl}):

\begin{theorem}
For every elliptic curve $E/\Q$ and number field $L\subset \Q(E[3^\infty])$,
$$
  (-1)^{\rksel EL3} = w(E/L).
$$
\end{theorem}

\begin{theorem}
Let $p$ be an odd prime, and $m,n\ge 1$ integers.
For every elliptic curve $E/\Q$ and number field
$F\subset \Q(\mu_{p^n},\sqrt[p^n]m)$,
$$
  (-1)^{\rksel EFp} = w(E/F).
$$
\end{theorem}

\begin{theorem}
Let $p\equiv3\!\!\mod4$ be a prime number and
$E/\Q$ an elliptic curve.
Suppose $F$ is a $p$-extension of an abelian extension of $\Q$,
Galois over~$\Q$. Then
$$
  (-1)^{\blangle\tau,\Xp EF\brangle} = w(E,\tau)
$$
for every orthogonal representation $\tau$ of $\Gal(F/\Q)$; in other words,
the \hbox{$p$-parity} conjecture holds for all orthogonal twists,
and in particular over all subfields of $F$.
\end{theorem}

Finally, we turn to local root numbers of elliptic curves.
The definition of these is not constructive,
but they have been classified in many cases.
For instance, $w(E/K_v)=-1$ when $v$ is Archimedean or when
$E/K_v$ has split multiplicative reduction, and $w(E/K_v)=+1$
when $E/K_v$ has good or non-split multiplicative reduction. Thus,
if $E/K$ is semistable,
$$
  w(E/K) = (-1)^{\#\{v|\infty\text{ in }K\}\>+\>\#\{\text{$v$ split mult. for $E/K$}\}}.
$$
At places of additive reduction,
an explicit classification has been given by
Rohrlich \cite{RohG} for $v\nmid 2,3$,
Halberstadt \cite{Hal} for $K=\Q$,
and Kobayashi \cite{Kob} for~$v|3$.
These express the root number in terms of the standard invariants of $E/K$,
making it easy to compute, as in the semistable case.
When $v|2$, the existing formulae for $w(E/K_v)$ (\cite{Whi, Root2})
are complicated and harder to use, as they rely on epsilon-factors
of Galois characters.

{}From the proof of the global root number formula of Theorem \ref{igloroot}
we will extract a uniform formula for the local root numbers.
It is slightly more cumbersome to state
as it involves all three 2-isogenous curves:

\begin{theorem}[see \S\ref{slocal}]
\label{ithmloc}
Let $E$ be an elliptic curve over a
finite extension~$\K$ of $\Q_p$.
Fix an invariant differential $\omega$ for $E/\K$ and write%
$$
  \OO E\K l =
    \leftchoicestretch{0.85}
      {+1}{\text{\rm if $C(E/\K,\omega)$ has even $l$-adic valuation,}}
      {-1}{\text{\rm if $C(E/\K,\omega)$ has odd $l$-adic valuation.}}
$$
Denote $\F\!=\!\K(E[2])$ and $d\!=\![\F:\K]$.
Let $E', E''$ and $E'''$ be the three curves 2-isogenous to $E$,
defined over suitable intermediate fields $\L', \ldots$  of $\F/\K$;
when $d\!=\!2$ we select $E'$ to be the one defined over $\K$.
Define $\smash{\OO {E'}\F l}$ etc. in the same way, computed with respect to
the differentials on $E', E'', E'''$ that pull back to $\omega$ under the
corresponding 2-isogenies.

If $p=2$, then $(-1)^{[\K:\Q_2]}w(E/\K)$ is given by
$$
\begin{array}{lllllll}
          \OO E\K2 \OO{E'}\K2 \OO{E''}\K2 \OO{E'''}\K2 &&\text{if}& d=1,\cr
          \OO E\K2 \OO{E'}\K2 \OO{E'}\F2 \OO{E''}\F2 &&\text{if}& d=2,\cr
          \OO E\F2 \OO{E'}\F2 &&\text{if}& d=3,\cr
          \OO E{\smash{\L'}}2 \OO{E'}{\smash{\L'}}2 \OO E\F2 \OO{E'}\F2 \OO E{\K(\smash{\sqrt{\Delta_E})}}3 \OO E\F3 &&\text{if}& d=6.\cr
\end{array}
$$
If $p\ne 2$, the same formulae compute $w(E/\K)$.
\end{theorem}

\begin{acknowledgements}
We would like to thank A. Bartel, D. Cain, B.~D.~Kim and A. J. Scholl
for many helpful discussions.
The first author is supported by a Royal Society
University Research Fellowship.
\end{acknowledgements}

\medskip

\subsection{Notation}
\label{ssnotation}
Number fields are denoted by $K, F,$ etc., and
local fields of characteristic 0 by $\K, \F$, etc.;
$F/K$ and $\F/\K$ are usually Galois extensions.
The product $\prod_v\cdots$ always refers to a product over all
places $v$ of $K$.

\noindent
\begin{tabular}{lll}
\vphantom{$\int^X$}
&$|a|_\K$          & normalised absolute value of $a$ in $\K$.\cr
&$(a,\F/\K)$       & Artin symbol of $a$ in $\F/\K$. In all our cases it will have\cr
&                  & order 1 or 2 in $\Gal(\F/\K)$, and we will regard it as $\pm 1$.\cr
&$(a,b)_\K$        & Hilbert symbol in $\K$ (the same as $(a,\K(\sqrt b)/\K)$).\cr
& $\blangle\cdot,\cdot\brangle$ & usual inner product of (characters of) representations.\\[2pt]
\end{tabular}

\newpage
\noindent
Notation for an elliptic curve $E/K$ or $E/\K$:

\noindent
\begin{tabular}{lll}
\vphantom{$\int^X$}
&$\rk E/K$         & Mordell-Weil rank of $E/K$, i.e. the rank of $E(K)/$torsion. \cr
&$\Xp EK$          & $\Hom_{\scriptscriptstyle{\Z_p}}\!(\varinjlim\Sel_{p^n}(E/K), \Q_p/\Z_p)\tensor\Q_p$, the dual $p^\infty$-Selmer. \cr
&$\rksel EKp$      & $\dim_{\Q_p}\Xp EK$, the $p^{\infty}$-Selmer rank of $E/K$ \cr
&                  & ($=\rk E/K+$ number of copies of $\Q_p/\Z_p$ in $\sha(E/K)$). \cr
&$w(E/k)$          & local root number of $E/k$ for $k$ local, or\cr
&                  & global root number $\prod_{v} w(E/k_v)$ for $k$ a number field.\cr
&$w(E/k,\!\tau)$   & local/global root number for the twist of $E$ by $\tau$ (see \cite{RohG}).\cr
&$C(E/\K,\omega)$ & $=c(E/\K) \!\cdot\! |{\omega}/{\neronnov{E}{\K}}|_\K^{}$ for $\K$ non-Archimedean, where $c$ is \cr
&   & the local Tamagawa number, and $\neronnov{E}{\K}$ a N\'eron differential;\cr
&   & $=\int_{E(\K)} |\omega|$ for $\K=\R$ \ and \ $2\int_{E(\K)} \omega\wedge \bar\omega$ for $\K=\C$.\cr
&   & ($\omega$ is a non-zero invariant differential for $E/\K$.)\cr
&$\CCbig EK$          & $\prod_{v} C(E/K_v,\omega)$ for any global invariant differential $\omega\ne 0$;\cr
&   & independent of the choice of $\omega$ by the product formula.\cr
&$E_\alpha$ & quadratic twist of $E$ by $\alpha$.\\[2pt]
\end{tabular}

For the definition and standard facts about root numbers we refer to
\cite{TatN}, \cite{RohG} and the appendix to \cite{Tamroot}. The latter also
contains a summary of basic properties of $\Xp EK$ that we occassionally use.

\section{2-parity over totally real fields}
\label{stotreal}

There is a large supply of elliptic curves
over number fields for which the $p$-parity conjecture is known to hold.
Over $\Q$ it is true
for all elliptic curves (\cite{Squarity} Thm. 1.4), and for odd $p$
for elliptic curves with non-integral $j$-invariant over totally real fields
(\cite{NekIV} Thm. 1).
The purpose of this section is to extend the latter result to $p=2$
(Theorem \ref{2parthm}).
Its proof goes along similar lines to that of Monsky \cite{Mon}
for elliptic curves over $\Q$ and uses potential modularity as in
\cite{NekIV}.

We refer to Wintenberger \cite{Win} for the definition of modularity
for elliptic curves over totally real fields. The two propositions below
are known
and are only included for completeness
(see Taylor's \cite{TayO} proof of Thm. 2.4 and \cite{TayR} proof of Cor. 2.2).

\begin{proposition}
\label{taylor}
Let $F/K$ be a cyclic extension of totally real fields, and $E/K$ an
elliptic curve. If $E/F$ is modular, then so is $E/K$.
\end{proposition}

\begin{proof}
Let $\pi$ be the cuspidal automorphic representation associated to $E/F$.
Pick a generator $\sigma$ of $\Gal(F/K)$.
Then $\pi^\sigma=\pi$ and therefore by Langlands'
base change theorem
 $\pi$ descends to a
cuspidal automorphic representation $\Pi$ over $K$.
Associated to $\Pi$ there is a compatible
system of representations $\rho_{\Pi,\lambda}$ (see \cite{Win}).
Let $\rho=\rho_{\Pi,\lambda}$ be one of these.
Its restriction to $\Gal(\bar F/F)$ agrees with
$V=V_l(E/F)\tensor_{\Q_l}\Q_\lambda$.
Because $E$ cannot have complex multiplication over $F$ (it is totally real),
$V_l(E/F)$ is absolutely irreducible, and therefore the only
representations that restrict to $V$ are
$W=V_l(E/K)\tensor_{\Q_l}\Q_\lambda$ and its twists by characters of
$\Gal(F/K)$.
Hence $\rho$ and $W$ differ by a 1-dimensional twist, whence $W$ is also
automorphic and $E/K$ is modular.
\end{proof}

\begin{proposition}
\label{redtomodular}
Let $E$ be an elliptic curve over a totally real field $K$, and $p$ a prime
number. If the $p$-parity conjecture for $E$ is true over every totally
real extension of $K$ where $E$ is modular, then it is true for $E/K$.
\end{proposition}

\begin{proof}
Because $E$ is potentially modular (\cite{Win} Thm. 1), there is a Galois
totally real extension $F/K$ over which $E$ becomes modular.
By Solomon's induction theorem,
there are soluble subgroups $H_i\<G=\Gal(F/K)$ and integers $n_i$, such that
$$
  \triv_G=\sum_i n_i\Ind_{H_i}^G\triv_{H_i},
$$
where $\triv$ denotes the trivial character.

Write $K_i$ for $F^{H_i}$.
Since $\Gal(F/K_i)=H_i$ is soluble, a repeated application
of Proposition \ref{taylor} shows that $E/K_i$ is modular.
By Artin formalism for $L$-functions,
$$
  L(E/K,s)=\prod_i L(E/K_i,s)^{n_i}.
$$
On the other hand,
writing $\XX=\Xp E{F}$, for every $H\<G$ we have
$$
  \rksel E{K^H}p=
  \dim\XX^H=\blangle\XX,\triv_H\brangle_H
    =\blangle\XX,\Ind_H^G\triv_H\brangle_G.
$$
The first equality is a standard fact about the  $p^\infty$-Selmer group,
see e.g. \cite{Squarity}~Lemma 4.14, and the last one
is Frobenius reciprocity.

By assumption, the $p$-parity conjecture holds for $E/K_i$.
Therefore
\beq
  \rksel EKp
  &=&\displaystyle
  \blangle\XX,\triv_G\brangle_G =
  \sum\nolimits_i n_i\rksel E{K_i}p\\[4pt]
  &\equiv&\displaystyle\sum\nolimits_i n_i\ord_{s=1}L(E/K_i,s)
  =\ord_{s=1}L(E/K,s) \mod 2.
\eeq
\end{proof}

\begin{remark}
\label{remmod}
The proof also shows that $L(E/K,s)$ has a meromorphic continuation to $\C$,
with the expected functional equation.
This is the same argument as in \cite{TayR}, proof of Cor. 2.2.
\end{remark}

\begin{theorem}
\label{2parthm}
Let $K$ be a totally real field, and $E/K$ an elliptic curve with
non-integral $j$-invariant. Then the $p$-parity conjecture holds for $E/K$
for every prime $p$.
\end{theorem}

\begin{proof}
For odd $p$ this is \cite{NekIV} Thm. 1, so suppose $p=2$.
Let $\p$ be a prime of~$K$ with $\ord_\p j(E)<0$. If $E$ has additive
reduction at $\p$, it becomes multiplicative over some totally real quadratic
extension $K(\sqrt\alpha)$, and the quadratic twist $E_\alpha/K$
has multiplicative reduction at $\p$ as well.
Because
\beq
  w(E/K(\sqrt\alpha))&=&w(E/K)w(E_\alpha/K)&&\text{and} \cr
    \rksel{E}{K(\sqrt\alpha)}2&=&\rksel EK2+\rksel {E_\alpha}K2,
\eeq
it suffices to prove the theorem for
elliptic curves with a prime of multiplicative reduction.
Since multiplicative reduction remains multiplicative in all extensions,
by Proposition \ref{redtomodular} we may also assume that $E$ is modular.

By Friedberg-Hoffstein's theorem \cite{FH} Thm. B,
there is a quadratic extension
$K(\sqrt \beta)$ of $K$ which is split at all bad primes for $E$ and such
that the quadratic twist $E_\beta$ has analytic rank $\le 1$. Since $E_\beta$ also
has a prime of multiplicative reduction, by Zhang's theorem (\cite{ZhaH} Thm. A)
$\sha(E_\beta)$ is finite and the Mordell-Weil rank of $E_\beta/K$ agrees
with its analytic rank;
in particular, the 2-parity conjecture holds for $E_\beta/K$. As it also
holds for $E/K(\sqrt \beta)$ by Kramer-Tunnell's theorem
(see Cor. \ref{kratuncor} below for a precise statement),
it is true for $E/K$.
\end{proof}

\section{Continuity of local invariants}
\label{scontinuity}

We gather a few basic facts concerning realisations of local fields
as completions of totally real fields and continuity of standard invariants.

\begin{lemma}
\label{totreal}
Let $\F/\K/\Q_p$ be finite extensions, with $\F/\K$ Galois.
There is a Galois extension of totally real number fields $F/K$ and a place
$v_0$ of $K$ with a unique place $w_0$ of $F$ above it,
such that $\K_{v_0}\iso \K$ and $\F_{w_0}\iso_\K\F$.
\end{lemma}

\begin{proof}
Say $\K=\Q_p[x]/(f)$ with monic $f\in\Q_p[x]$.
If $\tilde f\in \Q[x]$ is monic and $p$-adically close enough
to $f$, it defines the same extension of $\Q_p$ by Krasner's lemma. Now pick
such an $\tilde f$ which is also $\R$-close to any polynomial in $\R[x]$ whose
roots are real (weak approximation), and set $K_0=\Q[x]/(\tilde f)$
and $\p$ to be the prime above $p$ in $K_0$.

Pick a monic irreducible polynomial $g\in\K[x]$ whose splitting field is
$\F$, and again approximate it $\p$-adically with monic $\tilde g\in K_0[x]$
all of whose roots are totally real. Now let $F=K_0[x]/(\tilde g)$ and let $K$ be
the fixed field of $F$ under the decomposition group of some prime $w_0$
above $\p$.
\end{proof}

\begin{lemma}
\label{hilbcont}
For a local field $\K$ of characteristic 0,
the Hilbert symbol $(a,b)_\K$ is a continuous function of $a,b\in \K^*$.
\end{lemma}

\begin{proof}
Fix $a_0,b_0\in \K^*$, and set $\F=\K(\sqrt{a_0})$. Then
$\K(\sqrt{a})\iso \F$ for $a$ in an open neighbourhood $V$ of $a_0$ in $\K^*$.
The subgroup of norms $U=N_{\F/\K}\F^*$ is open in $\K^*$, and
$(a,b)_\K=(a_0,b_0)_\K$ for $a\in V, b\in b_0U$.
\end{proof}

\begin{proposition}
\label{continuity}
Let $\K$ be a finite extension of $\Q_p$ and
$E/\K$ an elliptic curve in Weierstrass form,
$$
  E: y^2 + a_1 xy + a_3y = x^3 +a_2 x^2 + a_4 x + a_6, \qquad a_i\in \K.
$$
There is an $\epsilon>0$ such that changing the $a_i$ to any $a'_i$
with $|a_i-a'_i|_\K<\epsilon$ does not change the conductor,
minimal discriminant, Tamagawa number, $C(E,\frac{dx}{2y+a_1x+a_3})$, the root number of $E$
and $T_l E$ as a $\Gal(\bar \K/\K)$-module for any given $l\ne p$.
\end{proposition}

\begin{proof}
The assertion for the conductor, minimal discriminant,
Tamagawa number and $C$ follows from
Tate's algorithm (\cite{TatA} or \cite{Sil2} IV.9).
The claim for $T_l E$ is a result of Kisin (\cite{Kis} p. 569),
and the root number is a function of $V_l E=T_l E\tensor\Q_l$.
Alternatively, that the root number is locally constant can be proved
in a more elementary way: see Helfgott's \cite{Hel} Prop. 4.2
when $E$ has potentially good reduction; in the potentially multiplicative
case it follows from Rohrlich's formula (\cite{RohG}~Thm.~2(ii)).
\end{proof}

\section{The Kramer--Tunnell conjecture}
\label{skratun}

In this section we recall the Kramer-Tunnell theory of local norm indices,
and prove the outstanding case of the conjecture of \cite{KT}
for local fields of characteristic 0 (Theorem \ref{kratunpf}).

\begin{notation}
For a separable quadratic (or trivial) extension $\F/\K$
of local fields and an elliptic curve $E/\K$, write
$$
  \kappa(E,\F/\K)=(-1)^{\dim_{\FF_2}(E(\K)/N_{\F/\K} E(\F))},
$$
where $N_{\F/\K}$ is the norm (or trace) map on points.
\end{notation}

\begin{remark}
\label{kappacont}
Alternatively, suppose $\K$ is non-Archimedean
with \hbox{$\vchar\K\!\ne\! 2$,}
and that $\F=\K(\sqrt\alpha)$ is a proper extension of $\K$.
Write $E$ and its quadratic twist $E_\alpha$ in simplified Weierstrass form,
$$
  E: y^2 = x^3+ax^2+bx+c, \qquad E_\alpha: y^2=x^3+\alpha ax^2+\alpha^2 bx+\alpha^3 c.
$$
By \cite{KT} Thm. 7.6,
$$
  \kappa(E,\F/\K)=(-1)^{\ord_2 C(E/\K,\tfrac{dx}y) C(E_\alpha/\K,\tfrac{dx}y) /
     C(E/\F,\tfrac{dx}y) }.
$$
In particular, by Proposition \ref{continuity},
$\kappa$ is a locally constant function of the coefficients $a,b,c$ of $E$.
\end{remark}

\begin{theorem}[Kramer \cite{Kra} Thm. 1]
\label{krathm}
Let $E/K$ be an elliptic curve over a number field,
and $F/K$ a quadratic extension. Then
$$
  (-1)^{\rksel EF2} = \prod_v \kappa(E,F_w/K_v),
$$
where the product is taken over the places $v$ of $K$, and $w$ is any
place of $F$ above $v$.
\end{theorem}

\begin{conjecture}[Kramer-Tunnell \cite{KT} Conj. 3.1]
\label{kratunconj}
Let $\K$ be a local field and $E/\K$ an elliptic curve. Let $\F/\K$ be a separable
quadratic extension, and write
$\chi: \Gal(\F/\K)\to\{\pm 1\}$ for the non-trivial character.
Let $\Delta_E\in \K^*$ be the discriminant of some model of $E/\K$. Then
$$
  w(E/\K)w(E/\K,\chi)(-\Delta_E,\F/\K) \conjeq \kappa(E,\F/\K).
$$
\end{conjecture}

\begin{theorem}[Kramer-Tunnell \cite{KT} p.315, \S4, \S5, \S6, \S8]
\label{kratunthm}
The conjecture holds when $\K$ is Archimedean, when $E/\K$ has
potentially multiplicative reduction, when $\K$ has odd residue
characteristic, when $\F/\K$ is unramified,
and when $E/\K$ has good reduction or reduction type \IV{} or \IVS.
\end{theorem}

\begin{corollary}[cf. \cite{KT} p. 351]
\label{kratuncor}
Let $E$ be an elliptic curve over a number field $K$, and $F/K$ a quadratic
extension. Then
$$
  (-1)^{\rksel EF2} = w(E/F),
$$
provided $F/K$ is unramified at those primes $v|2$ of $K$ where
$E$ has additive potentially good reduction not of type \IV, \IVS.
\end{corollary}

\begin{proof}
Let $\chi:\Gal(F/K)\to \{\pm1\}$ be the non-trivial character.
For a place $v$ of $K$ fix any $w|v$ in $F$, and write $\chi_v$
for the restriction of $\chi$ to $\Gal(F_w/K_v)$.
Observe that Conjecture \ref{kratunconj} also holds
when $\F=\K$ and $\chi=1$.
Thus,
\beq
  (-1)^{\rksel EF2} &\citeeq{\ref{krathm}}&
  \prod_v \kappa(E,F_w/K_v) \cr &\citeeq{\ref{kratunthm}}&
  \prod_v w(E/K_v)w(E/K_v,\chi_v)(-\Delta_E,F_w/K_v) \\[2pt] &\>\,\,=&
  w(E/K)w(E/K,\chi)\cdot 1 \\[2pt] &\>\,\,=&
  w(E/F),
\eeq
where the third equality is the product formula for the global root number and
for Artin symbols.
\end{proof}

\begin{theorem}
\label{kratunpf}
The Kramer-Tunnell conjecture holds over every local field~$\K$
of characteristic 0.
\end{theorem}

\begin{proof}
By Theorem \ref{kratunthm}, we may suppose $\K$ is non-Archimedean.
Let $\E/\K$ be the elliptic curve for which
we want to prove the conjecture.

Pick a totally real field $K$ with completion $\K$ at some place $v_0$
(Lemma~\ref{totreal}).
Choose also a quadratic extension $F\!=\!K(\sqrt\alpha)/K$ with
completion $\F$ above~$v_0$ such that
all places $v\ne v_0$ above 2 split in $F/K$
(weak approximation).

Let $E$ be an elliptic curve over $K$ that is sufficiently close to $\E$ at $v_0$,
so that $w(\E/\K)=w(E/\K)$, $w(\E_\alpha/\K)=w(E_\alpha/\K)$,
$(-\Delta_\E,\F/\K)=(-\Delta_E,\F/\K)$ and $\kappa(\E,\F/\K)=\kappa(E,\F/\K)$
(Proposition \ref{continuity}, Lemma \ref{hilbcont} and Remark \ref{kappacont}).
Moreover, we may assume
that at some place $v_1$ of $K$ which is split in $F/K$ the curve
$E$ has multiplicative reduction (weak approximation again).
By construction, it suffices to prove the claim for $E/K_{v_0}$.

By the Kramer-Tunnell theorem \ref{kratunthm}, the terms in the two products
$$
   \prod_v (-1)^{\dim_{\FF_2} \frac{E(K_v)}{N E(K_v(\sqrt\alpha))}}
   \quad\text{and}\quad
   \prod_v
    w(E/K_v)w(E_\alpha/K_v)(-\Delta_E,K_v(\sqrt\alpha)/K_v)
$$
agree except possibly at $v_0$.
The first product multiplies to $(-1)^{\rksel EF2}$
by Theorem \ref{krathm}. The second multiplies to $w(E/K)w(E_\alpha/K)$
which is also $(-1)^{\rksel EF2}$ by Theorem \ref{2parthm}
applied to $E/K$ and $E_\alpha/K$.
So the contributions from $v_0$ are also equal.
\end{proof}

\begin{corollary}
\label{corkratun}
Let $E$ be an elliptic curve over a number field $K$, and $F/K$ a quadratic
extension. Then
$$
  (-1)^{\rksel EF2} = w(E/F).
$$
In other words, the 2-parity conjecture holds for $E/F$.
\end{corollary}

\begin{proof}
Same as the proof of Corollary \ref{kratuncor}.
\end{proof}

\begin{remark}
As explained in \cite{Evilquad},
there are elliptic curves over number fields all of whose quadratic twists
have the same root number. As predicted by the parity conjecture, we now
find that these are precisely the curves all of whose quadratic twists have
the same parity of the $2^\infty$-Selmer rank.
In fact, for an elliptic curve $E$ over a number field $K$
the following conditions
\linebreak\newpage\noindent
are equivalent:
\begin{itemize}
\item[(a)] Every quadratic twist of $E/K$ has the same root number as $E$.
\item[(b)] $w(E/F)=w(E/K)^{[F:K]}$ for every finite extension $F/K$.
\item[(c)] Every quadratic twist of $E/K$ has the same parity of its
$2^\infty$-Selmer rank as $E$.
\item[(d)] Every quadratic twist of $E/K$ has the same parity of
the $\FF_2$-dimension of its $2$-Selmer group as $E$.
\item[(e)] $K$ has no real places, and $E$ acquires everywhere good reduction
   over some abelian extension of $K$.
\item[(f)] $K$ has no real places, and for all primes $p$ and all
  places \hbox{$v\nmid p$} of~$K$,
  the action of $\Gal(\bar K_v/K_v)$ on the Tate module $T_p(E)$ is abelian.
\end{itemize}

\noindent
To be precise,
(a)$\iff$(c) follows from Corollary \ref{corkratun};
(c)$\iff$(d) holds because $E$ and its
quadratic twists have the same 2-torsion, and the 2-primary part
of $\sha$
modulo its divisible part has square order;
(a)$\iff$(b)$\iff$(e)$\iff$(f) is proved in \cite{Evilquad} Thm. 1.
Also, a slightly weaker form of (a)$\iff$(d) is proved in \cite{MR3} \S 9.
\end{remark}

\section{Elliptic curves with a $p$-isogeny}
\label{sisogroot}

The purpose of the section is to prove the $p$-parity conjecture for
elliptic curves that admit a $K$-rational isogeny
of degree $p$, for $p=2$ and 3.

\begin{notation}
\label{sigmanot}
Let $\K$ be a local field of characteristic 0.
For an elliptic curve $E/\K$ with a $\K$-rational $p$-isogeny
$\phi: E\to E'$ write $\phi_\K:E(\K)\to E'(\K)$ for the induced
map on $\K$-rational points, and define
$$
 \sigma_\phi(E/\K) = (-1)^{\ord_p \frac{\#\coker\phi_\K}{\#\ker\phi_\K}} =
    (-1)^{\ord_p \frac{C(E'\!/\K,\omega')}{C(E/\K,\phi^*\omega')}},
$$
for any invariant differential $\omega'$ on $E'$.
(The second equality is an elementary computation using the usual filtration
on $E(\K)$;
e.g. it is spelled out for abelian varieties in \cite{SchaCl} Lemma 3.8.)
\end{notation}

\begin{theorem}[Cassels' formula]
\label{cassels}
Let $K$ be a number field and $E/K$ an elliptic curve
that admits a $K$-rational isogeny $\phi:E\to E'$ of prime degree~$p$.
Then
$$
  (-1)^{\rksel EKp} = \prod\nolimits_v \sigma_\phi(E/K_v).
$$
\end{theorem}

\begin{proof}
Birch \cite{Bir} p. 110, Fisher \cite{FisA}, Monsky \cite{Mon} Cor. 2.8
or \cite{Squarity} Rmk.~4.4.
\end{proof}

\begin{conjecture}
\label{isogrootconj}
Let $\K$ be a local field of characteristic~0 and $E/\K$ an elliptic curve
that admits a $\K$-rational isogeny $\phi:E\to E'$ of prime degree~$p$.
Write $\K(\ker\phi)$ for the field obtained by adjoining
the coordinates of points in $\ker\phi\subset E(\bar \K)$ to $\K$.
If $p\ne 2$, then
$$
  \hskip -10.5mm
  w(E/\K)\qequal\sigma_\phi(E/\K)\cdot(-1,\K(\ker\phi)/\K).
$$
If $p=2$, then for any model\,%
of $E$ over $\K$ of the form
$
  y^2 = x^3 + a x^2 + b x
$
with $\ker\phi=\{O,(0,0)\}$,
$$
  w(E/\K) \qequal \sigma_\phi(E/\K)
      \cdot
         \leftchoice
    {(a,-b)_\KK(-2a,a^2-4b)_\KK}{a\ne 0}
    {(-2,-b)_\KK}{a=0.}
$$
\end{conjecture}

\noindent
This conjecture is known in most cases:

\begin{theorem}[\cite{CFKS} Thm. 2.7, \cite{Isogroot} Thms. 5, 6, \S7]
\label{isogknown}
Conjecture \ref{isogrootconj} holds if
\begin{enumerate}
\item
$\K$ is Archimedean,
\item
$\K$ has residue characteristic $l\ne p$,
\item
$\K$ has residue characteristic $l\!=\!p\!>\!3$,
and either $E$ has potentially
ordinary reduction or $E$ achieves semistable reduction after
an abelian extension of $\K$,
\item
$\K$ has residue characteristic $l\!=\!p\!=\!3$, and $E$ is semistable.
\item
$\K$ has residue characteristic $l\!=\!p\!=\!2$, and $E$ has either good ordinary
or multiplicative reduction.
\end{enumerate}
\end{theorem}

\begin{corollary}
\label{oldisoparity}
Let $K$ be a number field and $E/K$ an elliptic curve
that \hbox{admits} a $K$-rational isogeny $\phi:E\to E'$ of prime degree~$p$.
Assume~that at primes $v|p$ of $K$, the curve $E/K_v$ satisfies the
relevant condition of \hbox{Theorem~\ref{isogknown}.} Then
$$
  (-1)^{\rksel EKp} = w(E/K).
$$
\end{corollary}

\begin{proof}
If $p=2$, choose a model for $E/K$ of the form $y^2=x^3+ax^2+bx$
with $\ker\phi=\{O,(0,0)\}$.
Now combine Theorems \ref{cassels} and \ref{isogknown}. The product of
Artin/Hilbert symbols over all places of $K$ is 1 by the product formula.
\end{proof}

Now we prove the conjecture for $p=2$ and 3 by a continuity argument.
First observe that for $p=2$ the Hilbert symbol term in
Conjecture \ref{isogrootconj} is continuous:

\begin{lemma}
\label{abcont}
Let $\K$ be a local field of characteristic 0. The function
$$
h(a,b)=
\leftchoice
    {(a,-b)_\KK(-2a,a^2-4b)_\KK}{a\ne 0}
    {(-2,-b)_\KK}{a=0.}
$$
is continuous for $a,b\in\K$ with $b\ne0$ and $a^2\!-\!4b\ne 0$.
\end{lemma}

\begin{proof}
For $a\ne 0$ this is Lemma \ref{hilbcont}. When $|a|_\K$ is small,
writing $(,)$ for $(,)_\K$,
\beq
  (a,-b)(-2a,a^2-4b)&=&
  (a,-b)(-2a,1-\tfrac{a^2}{4b})(-2a,-4b) \cr&=&
  (a,-b)(-2a,\square)(a,-4b)(-2,-4b)=(-2,-b),
\eeq
because elements of $\K$ close to 1 are squares.
\end{proof}

\begin{theorem}
\label{isogroot23thm}
Conjecture \ref{isogrootconj} holds for $p=2$ and $p=3$.
\end{theorem}

\begin{proof}
Any 2-isogeny $\phi: E\to E'$ of elliptic curves
over a field of characteristic not 2 or 3 has a model
\beq
  E_{a,b}&:& y^2=x^3+ax^2+bx, \cr
  E'_{a,b}&:& y^2=x^3-2ax^2+(a^2-4b)x,\cr
  \phi_{a,b}(x,y)&=&(x+a+bx^{-1}, y-byx^{-2}).
\eeq
Similarly, when $p=3$, there is a model
\beq
  E_{a,b}&:& y^2=x^3+a(x-b)^2, \cr
  E'_{a,b}&:& y^2=x^3+ax^2+18abx+ab(16a-27b), \cr
  \phi_{a,b}(x,y)&=&(x\!-\!4abx^{-1}\!+\!4ab^2x^{-3},y\!+\!4abyx^{-2}\!-\!8ab^2yx^{-3}).
\eeq
Conversely, for any $a, b$ in the ground field the formulae do define a
$p$-isogeny, provided the resulting curves are non-singular
(equivalently if $a(a^2\!-\!4b)\ne 0$ for $p=2$, and if
$ab(4a+27b)\ne0$ for $p=3$).

Say our elliptic curves and the $p$-isogeny are in this form,
with $a,b\in\K$.
Pick a totally real field $K$ with completion $\K$ at a place
$v_0$ (Lemma \ref{totreal}).
If $A, B\in K$ are $v_0$-close to $a, b$, then
$E_{A,B}$ and $E'_{A,B}$ are non-singular, so they are elliptic curves
over $K$ with a $p$-isogeny.
By Lemma \ref{abcont} and Proposition~\ref{continuity},
all the terms in the conjectural formula are the same for
$E_{A,B}/K_{v_0}$ as for $E_{a,b}/\K$.

Moreover, by the weak approximation theorem, we may choose $A, B$
so that the set of places where $E_{A,B}$ has multiplicative reduction
is non-empty and includes all $v\ne v_0$ above $p$.
To satisfy the latter, take for example $(A, B)$ to be $v$-close to
$(22,-7)$ for $p=2$ and $(1,-12)$ for $p=3$.
Because $E_{A,B}$ is defined over a totally real
field and has non-integral $j$-invariant,
$$
\begin{array}{ccl}
  \displaystyle\prod\nolimits_v w(E_{A,B}/K_v) &=& w(E_{A,B}/K)
    \>\>\>\citeeqref{2parthm}\>\>\>\displaystyle(-1)^{\rksel {E_{A,B}}Kp} \cr
    &\citeeqref{cassels}&\displaystyle \prod\nolimits_v \sigma_{\phi_{A,B}}(E_{A,B}/K_v)\\[5pt]
    &=&\displaystyle \prod\nolimits_v \sigma_{\phi_{A,B}}(E_{A,B}/K_v)\cdot
    \genfrac{(}{)}{0pt}{1}{\text{Artin/Hilbert symbol}}{\text{from Conjecture \ref{isogrootconj}}}.
\end{array}
$$
By Theorem \ref{isogknown}, the terms at $v$
in the first and the last expression agree for all $v\ne v_0$,
so they must agree at $v_0$ as well. Thus the conjecture holds for
$E_{A,B}/K_{v_0}$ and hence for $E_{a,b}/\K$.
\end{proof}

As in Corollary \ref{oldisoparity}, we deduce:

\begin{corollary}
\label{isogroot23}
Let $p=2$ or 3 and $K$ a number field.
The $p$-parity conjecture holds for every elliptic curve $E/K$
that admits a $K$-rational $p$-isogeny,
$$
  (-1)^{\rksel EKp} = (-1)^{\ord_p\frac{\CC EK}{\CC{E'\!}K}} = w(E/K),
$$
where $E'/K$ is the isogenous curve.
\end{corollary}

\section{Parity for twists coming from Brauer relations}
\label{stamroot}

The principal result of this section (Theorem \ref{tamrootBG}) is
a slight extension of the results of \cite{Tamroot} on the $p$-parity
conjecture for twists of elliptic curves.
First we illustrate it in the simplest possible setting,
when the Galois group is the symmetric group $\Sym_3$.
In fact, this is
the only case needed for the proof of Theorem \ref{imain} in \S\ref{sglobal}.


\subsection{Example: Galois group $G\protect\cong\Sym_3$}

As in \S\ref{skratun} and \S\ref{sisogroot}, we begin with a formula
expressing a Selmer rank in terms of local data.
Recall \cite{Squarity} Thm.~4.11 (with $p=3$):\,%
\footnote{
  The contributions from $v|\infty$ to
  ${\frac{\CC EF\CC EK^2}{\CC EM\CC EL^2}}$
  cancel when using the same $K$-rational $\omega$ over each field.
  The definition of $C$ in \cite{Squarity} excludes infinite places,
  so the
  formula there does not need the ${\frac{\CC EK^2}{\CC EL^2}}$ term,
  as it is then a rational square.}
\begin{theorem}
\label{s3rksel}
Let
$F/K$ be an $\Sym_3$-extension of number fields, $M$ and $L$~
inter\-mediate fields of degree 2 and 3 over $K$, and $E/K$ an
elliptic curve. Then
$$
  \rksel EK3+\rksel EM3+\rksel EL3 \equiv \ord_3\tfrac{\CC EF\CC EK^2}{\CC EM\CC EL^2} \mod 2.
$$
\end{theorem}

\noindent
We want to prove the parity conjecture for the left-hand side.
In other words, we claim that
\daggerequation{s3glo}{\hbox{$\dagger_{\text{\smaller[6]glo}}$}}{
  w(E/K)w(E/M)w(E/L) = \smash{(-1)^{\ord_3\frac{\CC EF\CC EK^2}{\CC EM\CC EL^2}}}.
}%
Both sides are products of local terms, and,
as in the Kramer-Tunnell case and the isogeny case,
we want to compare the contributions above each place $v$ of $K$.
It is elementary to check the following
(see e.g. \cite{Squarity} proof of Prop.~3.3):

\begin{fact}
\label{s3fact}
The contributions to each side of \eqref{s3glo} are trivial
whenever $v$ splits in $F$.
\end{fact}

The remaining case when $v$ does not split in $F$ leads to the following
purely local problem:

\begin{theorem}
\label{thms3loc}
Let $\F/\K$ be an $\Sym_3$-extension of $p$-adic fields, $\M$ and $\L$
intermediate fields of degree 2 and 3 over $\K$, and $\E/\K$ an
elliptic curve. Then
\daggerequation{s3loc}{\hbox{$\dagger_{\text{\smaller[6]loc}}$}}{
  w(\E/\K)w(\E/\M)w(\E/\L) = (-1)^{\ord_3\frac{C(\E/\F,\omega)}{C(\E/\M,\omega)}}.
}%
for any invariant differential $\omega$ for $E/\K$.
\end{theorem}

\begin{proof}
The formula is brutally worked out in \cite{Squarity} Prop. 3.3,
except for the case when $\K$ has residue characteristic 2 or 3,
$\E$ has additive reduction and $\F/\K$ is ramified.

Now we use a continuity argument to settle this remaining case.
Pick an $\Sym_3$-extension $F/K$ of totally real number fields with
completions $K_{v_0}=\K$ and $F_{v_0}=\F$ for some prime $v_0$ of $K$
(Lemma \ref{totreal}).
Choose an elliptic curve $E/K$ which is close enough $v_0$-adically to $\E$,
with semistable reduction at all $v\ne v_0$ above 2 and 3 and at least
one prime of multiplicative reduction.
By `close enough' we mean that the left- and the right-hand sides of
\eqref{s3loc} are the same for $E$ and $\E$ (Proposition \ref{continuity});
note that the right-hand side is independent of the choice of $\omega$.

By the 3-parity conjecture for $E$
over the intermediate fields of $F/K$ (Thm. \ref{2parthm})
and Theorem \ref{s3rksel}, we find that \eqref{s3glo} holds.
Since the terms in it agree above all $v\ne v_0$ by
\ref{s3fact} and the known cases of the formula,
they must agree at $v_0$ as well.
This proves \eqref{s3loc} for $E/K_{v_0}$ and hence for $\E/\K$ as well.
\end{proof}

As a corollary of \ref{s3fact} and \ref{thms3loc},
the formula \eqref{s3glo} holds for all elliptic curves
in all $\Sym_3$-extensions of number fields $F/K$.
Combining it with Theorem \ref{s3rksel}, we get

\begin{theorem}
\label{thms3glo}
Let $F/K$ be an $\Sym_3$-extension of number fields, and let
$M$ and $L$ be intermediate fields of degree 2 and 3 over $K$, respectively.
For every elliptic curve over $K$,
$$
  \smminusone^{\hbox{\smaller[4]$\rksel EK3\!+\!\rksel EM3\!+\!\rksel EL3$}} \!=\!
    \smash{\smminusone^{\ord_3\!\frac{\CC EF\CC EK^2}{\CC EM\CC EL^2}}}
  \!=\! w(\Esmslash K)w(\Esmslash M)w(\Esmslash L).
$$
\end{theorem}

\subsection{General case}

Now suppose $E/K$ is an elliptic curve, and $F/K$ any
Galois extension of number fields with Galois group $G$
(generalising $G=\Sym_3$ above).
The dual $p^\infty$-Selmer group $\Xp EF$ is a $\Q_pG$-representation,
and the $p$-parity conjecture for twists asserts that
$$
  (-1)^{\blangle \tau, \Xp EF\brangle} = w(E/K,\tau),
$$
for every self-dual (complex) representation $\tau$ of $G$.
There are twists for which the left-hand side can be expressed as a product
of local terms, and one may hope to prove the conjecture for them
by a place-by-place comparison.

This is what is done in \cite{Tamroot} for a class of such twists,
except that, as usual, the worst reduction cases are not considered there.
Knowing the $p$-parity conjecture over totally real fields allows us to
remove the constraints on the reduction for some of these twists,
via a continuity argument.

After recalling the set $\Ttp$ of twists for which
\cite{Tamroot} gives a local $p$-parity formula, we carry out
the continuity argument
(Theorem \ref{tamrootBG} and Corollary~\ref{gltamroot}).
We end with a list of examples and applications
in \S\ref{sstamappl}.

\medskip

Let $G$ be a finite group, and write $R_G$ for its (complex) representation
ring. Write $\cH$ for the set of subgroups of $G$ up to
conjugacy; its elements are in one-to-one correspondence
with transitive $G$-sets via $H\mapsto G/H$.
The {\em Burnside ring\/} of $G$ is the free abelian group $\Z\cH$, and
the map $H\mapsto \C[G/H]$ extends by linearity to a natural map
$\Z\cH \lar R_G$.
Following \cite{Tamroot}, elements in its kernel
are called {\em $G$-relations}.
Thus $\Theta=\sum n_i H_i$ is a $G$-relation if
$$
  \bigoplus\nolimits_i \C[G/H_i]^{\oplus n_i} \iso 0
$$
as a virtual representation.

Fix a prime $p$ and an identification $\bar\Q_p\!\iso\!\C$.
For a self-dual $\Q_p G$-represen\-tation $\rho$
and a $G$-relation $\Theta$, define the {\em regulator constant}
$$
  \RC_{\Theta}^{\Q_p}(\rho) = \RC_\Theta(\rho) = \prod\nolimits_i {\det(\tfrac{1}{|H|}\lara|\rho^{H})}^{n_i}
    \>\>\in\Q_p^*/\Q_p^{*2},
$$
where $\lara$ is any non-degenerate $G$-invariant $\bar\Q_p$-valued bilinear
pairing on~$\rho$, and $\det(\lara|V)$ denotes $\det(\langle v_i,v_j\rangle_{i,j})$
for any basis $v_i$ of $V$.
By \cite{Tamroot} \S2.2, $\RC_\Theta(\rho)$ is a well-defined element of $\Q_p^*/\Q_p^{*2}$
and is independent of $\lara$. We define
$$
  \Ttp = \left\{\begin{tabular}{c}
            self-dual $\bar\Q_p G$-representations $\tau$ such that\cr
            $\blangle\tau,\rho\brangle \equiv \ord_p \RC_\Theta(\rho) \!\!\!\mod 2$\cr
            for all self-dual $\Q_p G$-representations $\rho$.
         \end{tabular}\right\}
$$
With the identification $\bar\Q_p\iso\C$, we may consider $\Ttp\subset R_G$.

Let $F/K$ be a Galois extension of local or global fields
with Galois group~$G$, and $E/K$ an elliptic curve.
For a $G$-relation $\Theta=\sum n_i H_i$ for brevity we write
$$
  C(E,\Theta)=\prod_i C(E/F^{H_i},\omega)^{n_i}
$$
for any invariant differential $\omega$ for $E/K$.
Finally, if $K$ is a number field, we can decompose
$$
  C(E,\Theta) = \prod_v C_{w|v}(E,\Theta),
$$
where $v$ runs over the places of $K$, and $C_{w|v}(E,\Theta)$ picks all
terms in the product defining $C(E,\Theta)$ from all places above $v$.
(It is not hard to see that $C(E,\Theta)$ in the local case
and $C_{w|v}(E,\Theta)$ in the number field case are
independent of the choice of $\omega$,
see e.g. \cite{Tamroot}, proof of Cor. 3.4.)

According to \cite{Selfduality} Thms. 1.1, 1.5 (or \cite{Tamroot} Thm. 1.14),

\begin{theorem}
\label{squality}
Let $F/K$ be a Galois extension of number fields with Galois group $G$,
$p$ a prime and $\Theta$ a $G$-relation.
For every elliptic curve $E/K$,
the $\Q_p G$-representation $\Xp EF$ is self-dual, and
$$
\blangle\tau,\Xp EF\brangle
\equiv
\ord_p C(E,\Theta) \mod 2\>
\qquad \text{for all $\tau\in\Ttp$}.
$$
\end{theorem}

\noindent
In other words, the quotient of the `fudge factors' over the fields
defined by~$\Theta$ computes the parity of $\tau$ in the Selmer group.

\begin{notation}
Write $B_G$ for the subgroup of $R_G$ generated by the image of
the Burnside ring and all representations of the form $\sigma\oplus\sigma^*$.
\end{notation}

\begin{theorem}
\label{tamrootBG}
Let $\F/\K$ be a Galois extension of local fields of characteristic~0
with Galois group $G$. Let $\E/\K$ be an elliptic curve,
$p$ a prime and $\Theta$ a $G$-relation.
If $\E$ is semistable or $\K$ has residue characteristic $\ge 5$, then
$$
  w(\E,\tau) = (-1)^{\ord_p C(\E,\Theta)} \qquad\text{for all}\quad \tau\in\Ttp.
$$
In all cases, the formula holds for $\tau\in\Ttp\cap B_G$.
\end{theorem}

\begin{proof}
The first claim is a special case of \cite{Tamroot} Cor. 3.3.
Suppose therefore
that $\tau\in\Ttp\cap B_G$.

Let $F/K$ be a Galois extension of totally real fields
and~$v_0$ a prime of~$K$, such that $K_{v_0}=\K$ and $\Gal(F/K)=\Gal(\F/\K)$
(Lemma \ref{totreal}).
Let $E/K$ be an elliptic curve with non-integral $j$-invariant,
which is semistable at all primes $v\ne v_0$
above 2 and 3, and $E$ is sufficiently close to $\E$. By sufficiently close
we mean that $E/K_{v_0}$ and $\E/\K$ have the same local root numbers,
valuations of minimal discriminants and local Tamagawa numbers
in all intermediate fields of $\F/\K$
(possible by Proposition~\ref{continuity}).

By \cite{Tamroot} Cor. 3.4,
$$
  w(E/K_v,\Res_{G_v}\tau) = (-1)^{\ord_p C_{w|v}(E,\Theta)} \qquad \text{for all }v\ne v_0.
$$
Taking the product over all places $v$ we get that
the asserted formula at $v_0$ (and hence the claim for $\E/\K$) is equivalent to the global formula
$$
  w(E,\tau) \conjeq (-1)^{\ord_p C(E,\Theta)} \>\citeeqref{squality}\> (-1)^{\blangle\tau,\Xp EF\brangle}.
$$
But this is true, because
$\tau\iso(\sigma\oplus\sigma^*)\oplus\bigoplus_i\C[G/H_i]^{\oplus n_i}\in B_G$, and
$$
  w(E,\sigma\oplus\sigma^*)=1=(-1)^{\blangle\sigma\oplus\sigma^*,\Xp EF\brangle}
$$
by self-duality (\cite{Tamroot} Prop. A.2(1)), and
$$
  w(E,\C[G/H_i]) = w(E/F^{H_i}) = (-1)^{\rksel E{F^{H_i}}p} = (-1)^{\blangle\C[G/H_i],\Xp EF\brangle}
$$
by inductivity (\cite{Tamroot} Prop. A.2(2)) and the $p$-parity conjecture for
$E$ over the totally real field $F^{H_i}$ (Theorem \ref{2parthm}).
\end{proof}

\begin{corollary}
\label{gltamroot}
Let $F/K$ be a Galois extension of number fields with Galois group $G$.
Let $E/K$ be an elliptic curve, $p$ a prime and $\Theta$ a $G$-relation. Then
$$
  w(E,\tau) = (-1)^{\ord_p C(E,\Theta)} = (-1)^{\blangle\tau,\Xp EF\brangle}
$$
for all $\tau\in\Ttp$ such that $\Res_{G_v}\tau\in B_{G_v}$
for every prime $v|6$ where $E$ has additive reduction%
\footnote{In particular, this always applies to $\tau\in\Ttp\cap B_G.$}.
\end{corollary}

\begin{proof}
The first equality follows by the same argument as
the proof of Cor.~3.4 of \cite{Tamroot}.
(Essentially take the product over all places, and check that everything
behaves well under the passage to decomposition groups.)
The second equality is Theorem \ref{squality}.
\end{proof}

\begin{remark}
We expect Theorem \ref{tamrootBG} and
Corollary \ref{gltamroot} to hold for all $\tau\in\Ttp$, irrespectively
of the reduction type of $E$.
\end{remark}

\subsection{Examples and applications}
\label{sstamappl}

\begin{example}
The group $G=\Sym_3$ has four subgroups up to conjugacy,
$H=\{1\}, \Cy_2, \Cy_3, G$ and three irreducible representations:
$\triv$ (trivial), $\epsilon$ (sign) and $\rho$ (2-dimensional).
There is a $G$-relation
$$
  \Theta = \{1\} - 2\>\Cy_2 - \Cy_3 + 2\>G,
$$
unique up to multiples. One easily checks that
$\RC_\Theta^{\Q_p}(\triv) \!=\! \RC_\Theta^{\Q_p}(\epsilon) \!=\! \RC_\Theta^{\Q_p}(\rho) \!=\! 3,$
independently of the prime $p$. Thus,
$$
  \triv\oplus\epsilon\oplus\rho \in \Tau_{\Theta,3}.
$$
This representation lies in the Burnside ring of $G$,
$$
  \C[G/\Cy_2] \oplus \C[G/\Cy_3] \ominus \C[G/G] =
  (\triv\oplus\rho) \oplus (\triv\oplus\epsilon) \ominus \triv =
  \triv\oplus\epsilon\oplus\rho.
$$
By Corollary \ref{gltamroot}, we find that in an $\Sym_3$-extension of number
fields $F/K$ and an elliptic curve $E/K$,
$$
  w(E,\triv\oplus\epsilon\oplus\rho) = (-1)^{\ord_3 C(E,\Theta)} = (-1)^{\blangle\triv\oplus\epsilon\oplus\rho,\Xp EF\brangle}.
$$
Noting that
$$
  w(E,\triv\oplus\epsilon\oplus\rho) = w(E/K)w(E/F^{\Cy_3})w(E/F^{\Cy_2})
$$
and similarly for the Selmer rank by inductivity
(\cite{Tamroot} Prop. A.2(2)), we recover Theorem \ref{thms3glo}.
In the same way,
Theorem \ref{tamrootBG} recovers Theorem \ref{thms3loc}.
\end{example}

\begin{example}
Let $G=\Di_{2p}$ be the dihedral group of order $2p$
with \hbox{$p\equiv 3\mod 4$}. As for $\Sym_3$, there is a $G$-relation
$$
  \Theta = \{1\} - 2\>\Cy_2 - \Cy_p + 2\>\Di_{2p}.
$$
The irreducible $\C G$-representations are $\triv$ (trivial),
$\epsilon$ (sign) and $\sigma_1,...,\sigma_{\frac{p-1}2}$ (2-dimensional).
The irreducible $\Q_p G$-representations are
$\triv, \epsilon$ and \hbox{$\rho=\bigoplus_i\sigma_i$,} and their regulator
constants are (cf. \cite{Tamroot} Ex. 2.20)
$$
  \RC_\Theta^{\Q_p}(\triv) = \RC_\Theta^{\Q_p}(\epsilon) = \RC_\Theta^{\Q_p}(\rho) = p.
$$
Therefore $\tau=\triv\oplus\epsilon\oplus\rho\in\Ttp$ (as $\frac{p-1}2$ is odd);
it is also in the Burnside ring of $G$,
$$
  \C[G/\Cy_2] \oplus \C[G/\Cy_p] \ominus \C[G/G] =
  (\triv\oplus\rho) \oplus (\triv\oplus\epsilon) \ominus \triv =
  \tau.
$$
Now we apply Corollary \ref{gltamroot}:
\end{example}

\begin{proposition}
\label{propdih}
Suppose $F/K$ is a Galois extension of
number fields with Galois group $G\iso\Di_{2p}$ for a prime $p\equiv 3\mod 4$.
Then for every elliptic curve $E/K$ and every
2-dimensional irreducible $\C G$-representation~$\sigma$,
$$
  w(E,\sigma\oplus\triv\oplus\det\sigma) =
  (-1)^{\blangle \sigma\oplus\triv\oplus\det\sigma,\Xp EF\brangle}.
$$
\end{proposition}

\begin{proof}
By Corollary \ref{gltamroot} and the example above,
$$
  w(E,\triv\oplus\epsilon\oplus \rho) =
  (-1)^{\ord_p \frac{\CC EF\CC EK^2}{(\CC E{F^{\hbox{\tiny$\scriptscriptstyle\Cy_p$}}})
    (\CC E{F^{\hbox{\tiny$\scriptscriptstyle\Cy_2$}}})^2}} =
  (-1)^{\blangle \triv\oplus\epsilon\oplus\rho,\Xp EF\brangle}
$$
with $\epsilon$ and $\rho$ as in the example.
Now $\sigma$ is one of the constituents
of $\rho$, and $\det\sigma=\epsilon$. Because
$\Xp EF$ is a $\Q_p$-rational representation,
$$
  \blangle \rho,\Xp EF\brangle
   = \tfrac{p-1}2 \blangle \sigma,\Xp EF\brangle.
$$
On the other hand, by equivariance of local (and therefore global)
root numbers (\cite{RohI} Thms. 1, 2),
$$
  w(E,\rho)=w(E,\sigma)^{(p-1)/2}=w(E,\sigma),
$$
and the result follows.
\end{proof}

The proposition confirms \cite{Tamroot} Hypothesis 4.1
for elliptic curves when \linebreak $p\equiv 3\mod 4$. Thus, 
\cite{Tamroot} shows:

\begin{theorem}
Let $K$ be a totally real number field, $p\equiv 3\mod 4$ a prime, and
$E/K$ an elliptic curve.
Assume either that $K=\Q$ or that $E$ has non-integral $j$-invariant.
Suppose $F$ is a $p$-extension of an abelian extension of~$K$,
Galois over $K$. Then
$$
  (-1)^{\blangle\tau,\Xp EF\brangle} = w(E,\tau)
$$
for every orthogonal representation $\tau$ of $\Gal(F/K)$; in other words,
the \hbox{$p$-parity} conjecture holds for the twist of $E$ by $\tau$.
\end{theorem}

\begin{proof}
Combine \cite{Tamroot} Thm. 4.5 with the above proposition and
the $p$-parity theorems over $\Q$ and over totally real fields
(\cite{Squarity} Thm. 1.4 and \cite{NekIV} Thm.~1).
\end{proof}

\begin{theorem}
Let $K$ be a number field, $p\equiv 3\mod 4$ a prime, and
$E/K$ an elliptic curve.
Suppose $F$ is a $p$-extension of a Galois extension $M/K$, Galois over $K$.
If the $p$-parity conjecture
$$
  (-1)^{\rksel ELp} = w(E/L)
$$
holds for subfields $K\subset L\subset M$, then it holds
for subfields $K\subset L\subset F$.
\end{theorem}

\begin{proof}
Combine \cite{Tamroot} Thm. 4.3 with Proposition \ref{propdih}.
\end{proof}

\begin{theorem}
For every elliptic curve $\Esmslash \Q$ and number field $L\!\subset\!\Q(E[3^\infty])$,
$$
  (-1)^{\rksel EL3} = w(E/L).
$$
\end{theorem}

\begin{proof}
Same argument as in \cite{Tamroot} Ex. 4.8, using
\cite{Tamroot} Thm. 4.3 together with Proposition \ref{propdih},
$3$-parity for elliptic curves with a 3-isogeny
(Corollary~\ref{isogroot23})
and over abelian extensions of $\Q$ (\cite{Selfduality} Thm. 1.2).
\end{proof}

Another application of Corollary \ref{gltamroot} is for general $p$ but
$p$-adic towers where all orthogonal $\C G$-representations are in $B_G$.
As an example, we may take a `false Tate curve extension' over $\Q$:

\begin{example}
Let $E/\Q$ be an elliptic curve,
$p$ an odd prime, $m$ an integer which is not a perfect $p$th power,
and $F=\Q(\mu_{p^n},\sqrt[p^n]m)$ for some $n\ge 1$.
The complex irreducible representations of $G=\Gal(F/\Q)$ are
of the form $\chi\tensor\rho_j$ for a one-dimensional character $\chi$
of $\Gal(\Q(\mu_{p^n}/\Q))$ and
$$
  \rho_0 = \triv, \qquad
  \rho_j = \C[G/H_j] \ominus \C[G/H_{j-1}] \qquad (j=1,...,n),
$$
where $H_j=\Gal(F/\Q(\sqrt[p^j]m))$ is a subgroup of index $p^j$.
The self-dual ones are the $\rho_j$ and the non-trivial
character $\epsilon$ of $\Gal(\Q(\mu_p)/\Q)$ of order 2, so
all orthogonal representations of $G$ lie in $B_G$.
Thus the $p$-parity conjecture holds for twists of $E$
by all $\tau\in\Ttp$ for every $G$-relation $\Theta$.
Writing $U_j=\Gal(F/\Q(\mu_p,\sqrt[p^j]m))$,
one checks that
$$
  \triv\oplus\epsilon\oplus\rho_j \>\>\in\>\>
  \Tau_{U_j-U_{j-1}-(p\!-\!1)H_j+(p\!-\!1)H_{j-1}\,,\,p}.
$$
In particular, every orthogonal representation of $G$ is a sum of trivial,
sign representations of $G$ and those in $\Ttp$ for some $\Theta$.
Because the $p$-parity conjecture holds for $E/\Q$ and its quadratic twist
by $\epsilon$, it follows that
$$
  (-1)^{\blangle\tau,\Xp EF\brangle} = w(E,\tau)
$$
for every orthogonal representation $\tau$ of $G$.
This example removes the semistability restriction
from \cite{Squarity} Prop 4.13. See also \cite{CFKS} \S4 for related
results.
\end{example}

\section{Parity over fields and the global root number formula}
\label{sglobal}

Now we prove the two main results over number fields,
the parity conjecture assuming finiteness of $\sha$ and
the formula for the global root number.
The only two ingredients are the 2-isogeny theorem from
\S\ref{sisogroot} and the $\Sym_3$-example from \S\ref{stamroot}.
The argument is based on that of \cite{Squarity} Thm. 3.6.

\begin{theorem}
\label{main}
Let $E$ be an elliptic curve over a number field $K$, and
suppose that
$\sha(E/K(E[2]))$ has finite 2- and 3-primary parts. Then
$$
  (-1)^{\rk E/K} = w(E/K).
$$
\end{theorem}

\begin{proof}
Write $F=K(E[2])$, and note that $\Gal(F/K)\subset\GL_2(\FF_2)\iso\Sym_3$.
By assumption,
the 2- and 3-primary parts of $\sha(E/k)$ are finite
for $K\!\subset\! k\!\subset\! F$, see e.g. \cite{Squarity} Rmk. 2.10.

If $E$ has a $K$-rational 2-torsion point, the result follows from
Corollary~\ref{isogroot23}. If $F/K$ is cubic, then
the formula follows from that over $F$, as
both the parity of the rank and the root number
are unchanged in odd degree Galois extensions
(see e.g. \cite{Tamroot} Prop. A.2(3)).

The remaining case is $\Gal(F/K)\iso \Sym_3$.
Let $M$ be the quadratic extension of $K$ in $F$,
and let $L$ be one of the cubic ones. By the above argument,
\beq
  (-1)^{\rk E/M} = w(E/M)\qquad\text{and}\qquad
  (-1)^{\rk E/L} = w(E/L).
\eeq
On the other hand, by Theorem \ref{thms3glo},
$$
  (-1)^{\rk E/K+\rk E/M+\rk E/L} = w(E/K)w(E/M)w(E/L).
$$
\end{proof}

\noindent
The construction gives an explicit formula for the
global root number:

\begin{theorem}
\label{gloroot}
Let $E$ be an elliptic curve over a number field $K$.
Write $F=K(E[2])$ and $G=\Gal(F/K)$. Choose a non-trivial
2-torsion point $P$ of $E$, defined over $K$ if $G=\Cy_2$.
Write $E'=E/\{O,P\}$ for the 2-isogenous curve defined over $L=K(P)$.
Then
\beq
  w(E/K) = \left\{
  \begin{array}{ll}
     (-1)^{\ord_2\frac{\CC EL\CC EF}{\CC{E'}L\CC{E'}F}
     +\ord_3\frac{\CC EF\CC EK^2}{\CC E{K(\sqrt{\Delta_E})}\CC EL^2}}&\>\>\>G=\Sym_3,\cr
     (-1)^{\ord_2\frac{\CC EL}{\CC{E'}L}}&\>\>\>G\ne \Sym_3.\cr
  \end{array}
  \right.
\eeq
\end{theorem}

\begin{proof}
When $G=\Cy_1$ or $\Cy_2$, apply the 2-isogeny theorem
(Corollary \ref{isogroot23}).
If $G=\Cy_3$, then $w(E/K)=w(E/L)$ (\cite{Tamroot} Prop. A.2(3)),
and the result follows from that for~$E/L$.
Finally, suppose $G=\Sym_3$. Noting that $K(\sqrt{\Delta_E})$ is
the quadratic extension of $K$ in $F$,
\beq
  (-1)^{\ord_2\frac{\CC EL\CC EF}{\CC{E'}L\CC{E'}F}} &=& w(E/L)w(E/F),
     \cr
  (-1)^{\ord_3\frac{\CC EF\CC EK^2}{\CC E{K(\sqrt{\Delta_E})}\CC EL^2}}
    &=& w(E/L)w(E/F)w(E/K),
\eeq
by the 2-isogeny theorem over $L$ and $F$ and
by Theorem \ref{thms3glo}, respectively.
\end{proof}

\begin{remark}
It is already implicit in \cite{Squarity} that, assuming finiteness of $\sha$,
$(-1)^{\rk E/K}$ is given by the formula in the right-hand side
of Theorem \ref{gloroot}. It only relies on Cassels' formula \ref{cassels}
and Theorem \ref{s3rksel} and bypasses all comparisons with root numbers.
\end{remark}

\endcomment

\section{The local root number formula}
\label{slocal}

{}From the proof of the global root number theorem,
with a bit of extra work, one can also extract
a formula for the local root numbers (Theorem \ref{locmorloc} below).
The idea is to write the right-hand
side of Theorem \ref{gloroot} as a product of local terms and compare them
to the local root numbers. This is slightly delicate, as these local terms,
denoted by $\m(E/\K)$ below, must not depend on the global extension
$K(E[2])/K$. Also, to get a manageable relation between $w(E/\K)$ and
$\m(E/\K)$ we will make the construction symmetric in the three non-trivial
2-torsion points of $E$.

\begin{notation}
\label{orbnot}
Let $k$ be a field (later local or global) and $E$ an elliptic curve over $k$.
Write $S=E[2]\setminus\{O\}$, the set of non-trivial 2-torsion points
of $E/\bar k$. The symmetric group $\Sym_3$ acts on $S$. Denote the
irreducible representations of $\Sym_3$ by $\triv$ (trivial), $\epsilon$
(sign) and $\rho$ (2-dimensional).

For a group $G$ acting on $S$ (later a Galois group)
write $\Orb_{G}S$ for the set of $G$-orbits of $S$.
Abusing notation, we will write $P\in\Orb_{G}S$ for a representative of the
corresponding orbit, and $\triv, \epsilon, \rho$ for the restrictions
of these representations to $G$.
(Note that $\rho$ may become reducible.)

A point $P\in S$ gives rise to a 2-isogeny $\phi: E\to E/\{O,P\}$, and
we write $\sigma_\PP$ for $\sigma_{\phi}$ as defined in Notation \ref{sigmanot}.
\end{notation}

\begin{definition}
Let $\K$ be a local field of characteristic 0 and $E/\K$ an elliptic curve. Set
$\M=\K(\sqrt{\Delta_E})$, $\F=\K(E[2])$, $G=\Gal(\F/\K)$, $d=[\F:\K]$ and
\beq
  \m^{\triv\rho}(E/\K) &=& \displaystyle\prod_{P\in\Orb_{G}S} \sigma_\PP(E/\K(P)), \\[15pt]
  \m^{\triv\epsilon}(E/\K) &=& \displaystyle\prod_{P\in S}  \leftchoice
    {\sigma_\PPP(E/\F)}{2|d}{1}{2\nmid d}, \\[16pt]
  \m^{\triv\epsilon\rho}(E/\K) &=& \leftchoice
    {\smash{\hbox{$(-1)$}^{\ord_3 \frac{C(E/\F,\omega)}{C(E/\M,\omega)}}}}{d=6}
    {\hbox{$1\vphantom{1^X}$}}{d\ne 6,} \\[10pt]
  \m(E/\K) &=& \m^{\triv\rho}(E/\K)\m^{\triv\epsilon}(E/\K)\m^{\triv\epsilon\rho}(E/\K).
\eeq
\end{definition}

\begin{remark}
The $\m$ compute various
Selmer ranks of twists. If $E$ is defined over a number field $K$,
and $\triv, \epsilon$ and $\rho$ are the representations of
$G=\Gal(K(E[2])/K)\subset\Sym_3$ as in \ref{orbnot}, then
\begingroup
\def\v{{\smash{\hbox{\raise3pt\hbox{$\scriptstyle v$}}}}}
\beq
  \displaystyle\prod_\v \m^{\triv\rho}(E/K_v) &=&
    (-1)^{\blangle \triv\oplus\rho, \X2EF\brangle},\cr
  \displaystyle\prod_\v \m^{\triv\epsilon}(E/K_v) &=&
    (-1)^{\blangle \triv\oplus\epsilon, \X2EF\brangle},\cr
  \displaystyle\prod_\v \m^{\triv\epsilon\rho}(E/K_v) &=&
    (-1)^{\blangle \triv\oplus\epsilon\oplus\rho, \X3EF\brangle}.\cr
\eeq
\endgroup
This follows from the theorem below (combined with Corollary \ref{isogroot23}
and Theorem \ref{thms3glo}), although
it can also be deduced directly from Cassels' formula (Theorem \ref{cassels})
and Theorem \ref{s3rksel}.
\end{remark}

\begin{theorem}
\label{locmorloc}
Let $\K$ be a local field of characteristic zero, $E/\K$ an elliptic curve,
and $\triv, \epsilon, \rho$ the representations of
$G\!=\!\Gal(\K(E[2])/\K)$ as in~\ref{orbnot}.~Then
\beq
  &(1)&&\m^{\triv\rho}(E/\K) &=& w(E/\K)w(E/\K,\rho)\>(-1,-\Delta_E)_\K\,,\cr
  &(2)&&\m^{\triv\epsilon}(E/\K) &=& w(E/\K)w(E/\K,\epsilon)\>(-1,\Delta_E)_\K\,,\cr
  &(3)&&\m^{\triv\epsilon\rho}(E/\K) &=& w(E/\K)w(E/\K,\epsilon)w(E/\K,\rho)\,.\\[2pt]
\noalign{\noindent In particular,}
  &&&\vphantom{\int^{X^a}}\m(E/\K)&=&(-1,-1)_\K\> w(E/\K).
\eeq
\end{theorem}

\begin{proof}
Write $\F=\K(E[2])$. We will use $\Ind_{\L/\K}$ as a
shorthand for $\Ind_{\Gal(\F/\L)}^{\Gal(\F/\K)}$.
Recall that
local root numbers satisfy `self-duality', `inductivity in degree 0'
and the `determinant formula':
for a representation $\tau$ of $G$ and $\K\subset\L\subset\F$,
$$
  w(E/\K,\tau)=\overline{w(E/\K,\tau^*)},\qquad\qquad\qquad\qquad\>\> \eqno{(*)}
$$
$$
  w(E/\L)=w(\Ind_{\L/\K}\triv)^2\,w(E/\K,\Ind_{\L/\K}\triv),  \eqno{(\Ind)}
$$
$$
  w(\tau\oplus\tau^*) = (\det\tau)(-1),\qquad\qquad\qquad\qquad \eqno{(\det)}
$$
where $(\det\tau)(-1)$ is the character $\det\tau$
evaluated at the image of $-1$ under the local reciprocity map.
These are all well-known properties of root numbers;
see e.g. \cite{Tamroot} App. A.

\medskip
\noindent
(1) Choose a Weierstrass model $y^2=f(x)$ for $E/\K$. For $P=(r,0)\in S$,
define $a_\PP, b_\PP\in \K(P)$ by $f(x+r)=x^3+a_\PP x^2+b_\PP x$.
Note that
$$
  1\oplus\rho = \bigoplus_{P\in\Orb_G S} \Ind_{\K(P)/\K}\triv.
$$
Taking all products below over $P\in\Orb_{G}S$, we have
\beq
  \m^{\triv\rho}(E/\K) &=& \displaystyle\prod \sigma_\PP(E/\K(P))\cr
    &\citeeqref{isogroot23thm}& \displaystyle \prod w(E/\K(P))\>
      \leftchoice{(a_\PPP,-b_\PPP)_{\KKP}(-2a_\PPP,a_\PPP^2-4b_\PPP)_{\KKP}}{a_\PPP\ne 0}
        {(-2,-b_\PPP)_{\KKP}}{a_\PPP=0}
      \cr
    &\citeeqref{lem-1-1}& \displaystyle (-1,-1)_\K\> \prod w(E/\K(P))\cr
    &\rlap{$\!\!\!\!\!\citeeq{(\Ind)}$}& \displaystyle (-1,-1)_\K\> \prod w(\Ind_{\K(P)/\K}\triv)^2\, w(E/\K,\Ind_{\K(P)/\K}\triv)    \cr
    &=& \displaystyle (-1,-1)_\K\> w(\triv\oplus\rho)^2 w(E,\triv\oplus\rho)   \cr
    &\rlap{$\!\!\!\!\!\citeeq{(\det)}$}&
    \displaystyle (-1,-1)_\K\> \epsilon(-1) w(E,\triv\oplus\rho). \cr
\eeq
(2)
If $[\F:\K]$ is odd, then $\triv=\epsilon$ and the statement trivially holds.
Otherwise, by (Ind)
the right-hand side of the asserted formula is $w(E/\M)$,
where $\M=\K(\sqrt{\Delta_E})$ is the quadratic extension of $\K$ in $\F$.
This is the same as
$w(E/\F)$ because $\Gal(\F/\M)$ is either $\Cy_1$ or $\Cy_3$
(use (Ind), ($*$) and (det)). 
On the other hand, by Theorem \ref{isogroot23thm} and Lemma \ref{lem-1-1} the
left-hand side is
$w(E/\F)^3(-1,-1)_\F$. This is indeed $w(E/\F)$, as
$[\F:\K]$ is even and so $(-1,-1)_\F=1$ by Remark \ref{rem-1-1} below.

(3) If $[\F:\K]<6$, then the left-hand side is trivial, and the right-hand side
is of the form $w(E/\K,V\oplus V^*)$, which is also 1 by ($*$).
If $[\F:\K]=6$,
let $\M$ and $\L$ be a quadratic and a cubic extension of $\K$ in $\F$. Then
$$
  \m^{\triv\epsilon\rho}(E/\K) = (-1)^{\ord_3\frac{C(E/\F)}{C(E/\M)}}
    \>\>\citeeqref{thms3loc}\>\> w(E/\K)w(E/\M)w(E/\L)
$$
$$
  \citeeq{(\Ind)}
  w(E,\triv\oplus\epsilon\oplus\rho) w(\triv\oplus\epsilon\oplus\rho)^2
  \citeeq{(\det)}
  w(E,\triv\oplus\epsilon\oplus\rho).
$$
\end{proof}

\begin{remark}
\label{rem-1-1}
Recall that $(-1,-1)_\K$ measures whether the algebra of Hamiltonian
quaternions $\Q\langle 1,i,j,k\rangle$ is split or ramified over $\K$, so
$$
(-1,-1)_\K=\left\{\begin{array}{ll}
(-1)^{[\K:\R]}   &  \text{if $\K$ is Archimedean,}\cr
1 &                 \text{if $\K$ has odd residue characteristic,}\cr
(-1)^{[\K:\Q_2]} &  \text{if $\K$ has residue characteristic 2.}\cr
\end{array}
\right.
$$
In particular, Theorem \ref{locmorloc} directly
implies Theorem \ref{ithmloc}.
\end{remark}

The following lemma was used in the proof of the theorem:

\begingroup

\def\alpha{r}
\def\beta{s}
\def\gamma{t}

\begin{lemma}
\label{lem-1-1}
Suppose $\K$ is a local field of characteristic 0, and $f(x)\in \K[x]$
is monic of degree 3. Let $\alpha_1,\alpha_2,\alpha_3$ be its roots in $\bar \K$, and
$\F=\K(\alpha_1,\alpha_2,\alpha_3)$.
For $\alpha\in\{\alpha_1,\alpha_2,\alpha_3\}$ write $f(x+\alpha)=x^3+a_\alpha x^2+b_\alpha x$. Then
$$
  \prod_\alpha
  \leftchoice{(a_\alpha,-b_\alpha)_{\KKR}(-2a_\alpha,a_\alpha^2-4b_\alpha)_{\KKR}}
     {a_\alpha\ne 0}{(-2,b_\alpha)_{\KKR}}{a_\alpha=0}
   \quad=\quad (-1,-1)_\K,
$$
the product taken over some representatives of
the orbits of $G=\Gal(\F/\K)$ on $\{\alpha_1,\alpha_2,\alpha_3\}$.
\end{lemma}

\begin{proof}
In view of Lemma \ref{abcont}, we can assume that all $a_r\ne 0$.
Write $N$ for the norm map from $\F$ to $\K$.
Recall that the Hilbert symbol $(,)=(,)_\F$ is bilinear, symmetric,
and satisfies
\beq
  (a,1-a)=(a,-a)=1, \cr
  (a+b,-ab)=(1+\tfrac ba,-\tfrac ba)(a,-ab)=(a,b),\cr
  (a,b)(b,c)(a,c)(-1,abc)=(-1,-1) & \text{if }a+b+c=0,\cr
  (a,-b)(b,-c)(c,-a)=(-1,-1) & \text{if }a+b+c=0,\cr
  (a,b)(-a,-c)(b,-c)=1, & \text{if }a+b+c=0,\cr
 (a,Nb)_\K = (a,b) & \text{for }a\in \K, b\in \F,
\eeq
whenever the constituents are non-zero.

Case 1: $\F=\K$.
Say $\alpha=\alpha_1$. Then $a_\alpha^2-4b_\alpha=(\alpha_2-\alpha_3)^2$ is a square in $\K$, so
the second Hilbert symbol is trivial. The first one is
$$
  \bigl((\alpha_1-\alpha_2)+(\alpha_1-\alpha_3),-(\alpha_1-\alpha_2)(\alpha_1-\alpha_3)\bigr)
     =
  (\alpha_1-\alpha_2,\alpha_1-\alpha_3),
$$
and similarly for $\alpha=\alpha_2$ and $\alpha_3$. The total product is
\beq
  (\alpha_1-\alpha_2,\alpha_1-\alpha_3) \cdot (\alpha_2-\alpha_1,\alpha_2-\alpha_3) \cdot (\alpha_3-\alpha_1,\alpha_3-\alpha_2) = 
  (-1,-1).
\eeq

Case 2: $\alpha_1\in \K, \alpha_2,\alpha_3\notin \K$,
and the product is taken over $\alpha\in\{\alpha_1,\alpha_2\}$.
Let $\beta=\alpha_2-\alpha_1, \gamma=\alpha_3-\alpha_1$, and expand
\beq
  f(x+\alpha_1)=x(x-\beta)(x-\gamma)=x^3+(-\beta-\gamma)x^2+\beta\gamma x,\cr
  f(x+\alpha_2)=x(x+\beta)(x+\beta-\gamma)=x^3+(2\beta-\gamma)x^2+(\beta^2-\beta\gamma)x.
\eeq
The total product of Hilbert symbols is
\beq
  (\!-\!\beta\!-\!\gamma,\!-\!\beta\gamma)_\K \cdot (2\beta\!+\!2\gamma,(\beta\!-\!\gamma)^2)_\K \cdot (\beta\!+\!(\beta\!-\!\gamma),\!-\!\beta(\beta\!-\!\gamma)) \cdot (\!-\!4\beta\!+\!2\gamma,\gamma^2)=\cr
  (\!-\!\beta\!\!-\!\!\gamma,\!-\!\tfrac 14)_\K (\!-\!\beta\!\!-\!\!\gamma,N(2\beta))_\K \cdot
    (2\beta\!+\!2\gamma,\!-\!1)_\K(2\beta\!+\!2\gamma,N(\beta\!\!-\!\!\gamma))_\K \cdot (\beta,\beta\!\!-\!\!\gamma) \cdot 1 = \cr
  (-\beta-\gamma,-1)_\K (-\beta-\gamma,2\beta)
   (2,-1)_\K (\beta+\gamma,-1)_\K (2\beta+2\gamma,\beta-\gamma) (\beta,\beta-\gamma) = \cr
  (-1,-1)_\K (-\beta-\gamma,2\beta)(\beta+\gamma,\beta-\gamma)(2,\beta-\gamma)(\beta,\beta-\gamma) = \cr
  (-1,-1)_\K (-\beta-\gamma,2\beta)(\beta+\gamma,\beta-\gamma)(2\beta,\beta-\gamma) = (-1,-1)_\K. \cr
\eeq

Case 3: $[\F:\K]=3$. The claim follows from that for $f$ over $\F$,
because $(-1,-1)_\K=(-1,-1)_\F$ and the terms in the left-hand side
of the formula also remain the same for each $r$.

Case 4: $[\F:\K]=6$. Invoke the formula over the quadratic and over a cubic
intermediate field of $\F/\K$, and multiply the two.
\end{proof}

\endgroup


\begin{thebibliography}{29}

\bibitem{Bir}
B. J. Birch, Conjectures concerning elliptic curves,
Proc. Sympos. Pure Math., Vol. VIII (1965), Amer. Math. Soc., Providence,
R.I, 106--112.


\bibitem{CFKS}
J. Coates, T. Fukaya, K. Kato, R. Sujatha,
Root numbers, Selmer groups and non-commutative Iwasawa theory,
2007, to appear in J. Algebraic Geom.


\bibitem{Isogroot}
T. Dokchitser, V. Dokchitser,
Parity of ranks for elliptic curves with a cyclic isogeny,
J. Number Theory 128 (2008), 662--679.

\bibitem{Squarity}
T. Dokchitser, V. Dokchitser,
On the Birch--Swinnerton-Dyer quotients modulo squares,
2006, arxiv: math.NT/0610290, to appear in Annals of Math.

\bibitem{Selfduality}
T. Dokchitser, V. Dokchitser,
Self-duality of Selmer groups,
Math. Proc. Cam. Phil. Soc. 146 (2009), 257--267.

\bibitem{Tamroot}
T. Dokchitser, V. Dokchitser,
Regulator constants and the parity conjecture,
2007, arxiv: 0709.2852, to appear in Invent. Math.

\bibitem{Root2}
T. Dokchitser, V. Dokchitser,
Root numbers of elliptic curves in residue characteristic~2,
Bull. London Math. Soc. 40 (2008), 516--524.

\bibitem{Evilquad}
T. Dokchitser, V. Dokchitser,
Elliptic curves with all quadratic twists of positive rank,
Acta Arith. 137 (2009), 193--197.

\bibitem{FisA}
T. Fisher, Appendix to V. Dokchitser,
Root numbers of non-abelian twists of elliptic curves,
Proc. London Math. Soc. (3) 91 (2005), 300--324.

\bibitem{FH}
S. Friedberg, and J. Hoffstein, Nonvanishing theorems for
automorphic L-functions on $\GL(2)$,
Annals of Math. 142 (2), 1995, 385--423.



\bibitem{Hal}
E. Halberstadt, Signes locaux des courbes elliptiques en 2 et 3,
C. R. Acad. Sci. Paris S\'erie I Math. 326 (1998), no. 9, 1047--1052.

\bibitem{Hel}
H. A. Helfgott, On the behaviour of root numbers in families of
elliptic curves, arXiv: math/0408141v3.


\bibitem{Kis}
M. Kisin, Local constancy in $p$-adic families of Galois representations,
Math. Z., 230 (1999), 569--593.

\bibitem{Kob}
S. Kobayashi, The local root number of elliptic curves with wild
ramification, Math. Ann. 323 (2002), 609--623.


\bibitem{Kra}
K. Kramer, Arithmetic of elliptic curves upon quadratic extension,
Trans. Amer. Math. Soc. 264 (1981), 121--135.

\bibitem{KT}
K. Kramer, J. Tunnell, Elliptic curves and local $\epsilon$-factors,
Compos. Math. 46 (1982), 307--352.

\bibitem{MR}
B. Mazur, K. Rubin, Finding large Selmer ranks via an arithmetic theory
of local constants,
Annals of Math. 166 (2), 2007, 579--612.

\bibitem{MR3}
B. Mazur, K. Rubin,
Ranks of twists of elliptic curves and Hilbert's tenth problem,
preprint, 2009, arxiv: 0904.3709.

\bibitem{Mon}
P. Monsky, Generalizing the Birch--Stephens theorem. I: Modular curves,
Math. Z., 221 (1996), 415--420.




\bibitem{NekIV}
J. Nekov\'a\v r, On the parity of ranks of Selmer groups IV,
to appear in Compos. Math.


\bibitem{RohG}
D. Rohrlich, Galois Theory, elliptic curves, and root numbers,
Compos. Math. 100 (1996), 311--349.

\bibitem{RohI}
D. Rohrlich, Galois invariance of local root numbers, preprint, 2008,
{http://math.bu.edu/people/rohrlich/invariance.pdf.}

\bibitem{SchaCl}
E. Schaefer, Class groups and Selmer groups,
J. Number Theory 56 (1996), no. 1, 79--114.

\bibitem{Sil2}
J. H. Silverman, Advanced Topics in the Arithmetic of Elliptic Curves,
GTM 151, Springer-Verlag 1994.

\bibitem{TatC}
J. Tate, On the conjectures of Birch and Swinnerton-Dyer and a
geometric analog, S\'eminaire Bourbaki, 18e ann\'ee, 1965/66, no. 306.

\bibitem{TatA}
J. Tate, Algorithm for determining the type of a singular fiber in
an elliptic pencil, in: Modular Functions of One Variable IV,
Lect. Notes in Math. 476, B. J. Birch and W. Kuyk, eds., Springer-Verlag,
Berlin, 1975, 33--52.

\bibitem{TatN}
J. Tate, Number theoretic background, in: Automorphic forms,
representations and L-functions, Part 2
(ed. A. Borel and W. Casselman), Proc. Symp. in Pure Math.
33 (AMS, Providence, RI, 1979) 3--26.


\bibitem{TayO}
R. Taylor, On icosahedral Artin representations II,
American Journal of Mathematics 125 (2003), 549--566.

\bibitem{TayR}
R. Taylor, Remarks on the conjecture of Fontaine and Mazur,
J. Inst. Math. Jussieu~1, 2002, 1--19.



\bibitem{Whi}
D. Whitehouse, Root numbers of elliptic curves over 2-adic fields, preprint,
2006,
http://www-math.mit.edu/~dw/maths/elliptic2.pdf.

\bibitem{Win}
J.-P. Wintenberger, Potential modularity of elliptic curves over totally
real fields, appendix to \cite{NekIV}.


\bibitem{ZhaH}
S. Zhang, Heights of Heegner points on Shimura curves,
Annals of Math. 153 (2001), 27--147.


\end{thebibliography}
\end{document}